\documentclass[12pt,reqno]{amsart}
\usepackage{amsmath, amsfonts, amsthm, amssymb, graphicx, cite}

\theoremstyle{plain}
\newtheorem{theorem}{Theorem}[section]

\theoremstyle{definition}

\newtheorem{example}{Example}[section]
\newtheorem{remark}{Remark}[section]

\newcommand{\e}{\varepsilon}

\renewcommand{\div}{\text{div}}
\renewcommand{\O}{\Omega}

\newcommand{\R}{{\mathbb R}}

\def\be{\begin{equation}}
\def\ee{\end{equation}}
\def\bes{\begin{equation*}}
\def\ees{\end{equation*}}
\def\bc{\begin{cases}}
\def\ec{\end{cases}}

\numberwithin{equation}{section}

\begin{document}

\title[Fluids, Geometry, and Navier-Stokes Turbulence]
{Fluids, Geometry, and the Onset of Navier-Stokes Turbulence in Three Space Dimensions}

\author{Gui-Qiang Chen}
\address{Mathematical Institute, University of Oxford, Oxford OX2 6GG, UK;
AMSS \& UCAS, Chinese Academy of
Sciences, Beijing 100190, China.}
\email{chengq@maths.ox.ac.uk}

\author{Marshall Slemrod}
\address{Department of Mathematics, University of Wisconsin, Madison, WI 53706, USA.}
\email{slemrod@math.wisc.edu}

\author{Dehua Wang}
\address{Department of Mathematics, University of Pittsburgh, Pittsburgh, PA 15260, USA.}
\email{dwang@math.pitt.edu}

\dedicatory{\it Dedicated to our friend Edriss Titi on the occasion of his 60th birthday}


\date{\today}

\begin{abstract}
A theory
for the evolution of a metric $g$~ driven by
the equations of three-dimensional continuum mechanics is developed.
This metric in turn allows for the local existence of an evolving three-dimensional
Riemannian manifold immersed in the six-dimensional Euclidean space.
The Nash-Kuiper theorem is then applied to this Riemannian manifold
to produce a {\it wild} evolving $C^{1}$ manifold.
The theory is applied to the incompressible Euler and Navier-Stokes equations.
One practical outcome of the theory is a computation of~ critical profile initial data
for what may be interpreted as the onset of turbulence for the classical
incompressible Navier-Stokes equations.
\end{abstract}
\maketitle


\section{Introduction} \label{Introduction}

The purpose of this paper is to continue our previous study of the link
between the equations of inviscid continuum mechanics and the motion of Riemannian manifolds ({\it cf.} \cite{ACSW}).
More specifically, we have shown in \cite{ACSW} that a solution of the system of balance laws of mass and momentum
in two space dimensions can be mapped into an evolving two-dimensional Riemannian manifold in $\mathbb{R}^{3}$.
Furthermore, it is shown that the geometric image of smooth solutions of the continuum equations for non-wild
data (not simple shears) can be shadowed by a non-smooth geometric motion.
In addition,  the geometric initial value problem for these non-smooth solutions
has an infinite number
of energy preserving non-smooth solutions.
Since the earlier paper  \cite{ACSW} was focused on two-dimensional continuum mechanics, it is natural to
develop a theory to deal with three space dimensional case, and we provide this here when the Riemannian manifold
is now time evolving in $\mathbb{R}^{6}$.
Moreover, our earlier paper only dealt with inviscid materials.
In this paper, we extend our results to viscous fluids, including the incompressible viscous fluids
governed by the classical Navier-Stokes equations.
In particular, we use the geometric theory to predict the critical profile initial data
for the onset of turbulence in channel flow.

The main value of such a continuum-geometry link in \cite{ACSW} is to give a rather straightforward
demonstration for the existence of wild solutions for the equations of inviscid continuum mechanics
for classical incompressible and compressible fluids and neo-Hookian elasticity, and for the non-uniqueness
of weak entropy solutions to the initial value problem.
The work was motivated by the important sequence of papers \cite{CDS2012,DS2009,DS2010,DS2012} by DeLellis,
Sz\'{e}kelyhidi Jr.,
and others
on the application of Gromov's $h$-principle and convex integration \cite{Gromov1986} to provide
both the existence of wild solutions and the non-uniqueness of solutions to the initial value problem.
Since Gromov's work is based on
the classical Nash-Kuiper theorem for non-smooth isometric embeddings in Riemannian geometry,
it was our goal to exposit a direct map from continuum mechanics to the motion of a Riemannian manifold,
in order to
avoid sophisticated analysis needed to apply Gromov's theory.
Furthermore, in making the direct continuum to geometry link, it becomes apparent that
our approach is very much in the spirit of the Einstein equations of general relativity,
{\it i.e.},  in both our work and general relativity,
the matter relation (fluids, gases, {\it etc.}) on one side of the equations
drives the geometric motion of an evolving space-time metric.

Perhaps in retrospect, it is no surprise that the proof of the continuum to geometry map
in two space dimensions is distinctly different than the proof we provide here
for three space dimensions.
The reason is more than just technical and lies at the heart of much of the work
for the isometric embedding problem in three and higher dimensional Riemannian manifolds.
More specifically, for two-dimensional manifolds, it is rather straightforward to analyze
the Gauss-Codazzi equations that provide both necessary and sufficient conditions
for the existence of an isometrically embedded manifold in $\mathbb{R}^{3}$,
which was the view we took in \cite{ACSW}.
However, for the case of three space dimensions, all the work to date has avoided the application
of the next level of necessary and sufficient conditions (the Gauss, Codazzi, and Ricci equations)
and has dealt with the fully nonlinear embedding equations directly;
we follow this direct approach here as well.
We do this by invoking two key results: The solvability of the system for determination of
a metric $g$, given the components of the Riemann curvature tensor ({\it cf.} DeTurck-Yang \cite{DeYang}),
and the local solvability of the isometric embedding problem for the embedding of a three-dimensional
Riemannian manifold into $\mathbb{R}^{6}$ ({\it cf.}  Maeda-Nakamura \cite{NakamuraMaeda86, NakamuraMaeda89};
see also Goodman-Yang \cite{GoodmanYang} and Chen-Clelland-Slemrod-Yang-Wang \cite{CCSWY}).

This paper is divided into six sections after this introduction.
Section 2 provides a review of the elements of Riemannian geometry and
the isometric embedding problem.
A short presentation of the balance laws of continuum mechanics is given in Section 3.
In Section 4, a short proof is presented for deriving the metric map taking the continuum mechanical evolution
into the evolution of a three-dimensional Riemannian manifold immersed in the six-dimensional Euclidean space.
In Section 5, we present a short observation that a shearing flow may be mapped into a Riemannian flat manifold
in the six-dimensional Euclidean space.
In Section 6,  the Nash-Kuiper theorem for non-smooth isometric embeddings is first recalled
and is then applied
to show that the geometric initial value problem has an infinite set of weak solutions
for the same non-smooth initial data.
Finally, in Section 7, the theory is applied to the classical incompressible Navier-Stokes equations
and gives a prediction of the critical profile initial data for the onset of turbulence in channel flow.
The critical profile predicted by the geometric theory of this paper
is in agreement with the experimentally observed profile given in Reichardt \cite{Reichardt}.

\bigskip
\section{Geometry and isometric embedding}\label{GeometryEmbedding}

This section is devoted to some preliminary discussion and review about geometry and isometric embedding.

\subsection{Geometry} \label{Geometry}
Riemann \cite{Riemann} in 1854 introduced the notion of an abstract manifold with metric structure,
motivated by the problem of defining a surface in a Euclidean
space independently of the underlying Euclidean space.
The isometric embedding problem seeks to establish the conditions for the Riemannian manifold to be a
sub-manifold of a Euclidean space with the same metric.
For example, consider a smooth $n$-dimensional Riemannian manifold $M^{n}$ with metric tensor $g$.
In terms of local coordinates $x_{i},\;i =1,2 \ldots  n$,
the distance on $M^{n}$ between neighbouring points is
$$
ds^{2} =g_{i j}  dx_{i}  dx_{j}, \qquad i,j =1,2,\ldots  n,
$$
where and
throughout the paper, the standard summation convention is adopted.
Now let $\mathbb{R}^{m}$ be the $m$-dimensional Euclidean space, and let $y :M^{n} \rightarrow \mathbb{R}^{m}$ be
a smooth map so that the distance between neighbouring points is given by
$$
d \bar{s}^{2} = dy\cdot dy =y_{,j}^{i} y_{,k}^{i}  dx_{j}  dx_{k},
$$
where the subscript comma denotes partial differentiation with respect
to the local coordinates $x_{i}, i=1,2,\ldots, n$.
The existence of a \textit{global embedding} of $M^{n}$ in $\mathbb{R}^{m}$ is equivalent to the existence of
a smooth map $y$ for each $x \in M^{n}$ into $\mathbb{R}^{m}$.
\textit{Isometric embedding} requires the existence of a map $y$ for which the distances are equal.
That is,
$$
g_{i j}  dx_{i}  dx_{j} =y_{,j}^{i} y_{,k}^{i}  dx_{j}  dx_{k},
\quad\mbox{or}\quad
y_{,j}^{i} y_{,k}^{i} =g_{j k},
$$
which may be compactly rewritten
as
$$
\partial_{i}y \cdot  \partial_{j}y =g_{i j},
$$
where $\partial_{i} =\frac{\partial}{\partial x_{i}}$,
and the inner product in $\mathbb{R}^{m}$ is denoted by symbol ``$ \cdot $''.

The classical isometric embedding of a $2$-dimensional Riemannian manifold
into the $3$-dimensional Euclidean space is comparatively well studied and
comprehensively discussed in Han-Hong \cite{HanHong}. By contrast,
the embedding of $n$-dimensional Riemannian manifolds into the $J_n:=\frac{n(n +1)}{2}$-dimensional Euclidean space
has only a comparatively small literature.
When $n =3$, the main results are due to Bryant-Griffiths-Yang \cite{BGY}, Nakamura-Maeda \cite{NakamuraMaeda86}, Goodman-Yang \cite{GoodmanYang},
and most recently to Poole \cite{Poole} and Chen-Clelland-Slemrod-Wang-Yang \cite{CCSWY}.
A general, related case when $n \geq 3$ is considered in Han-Khuri  \cite{HanKhuri}.
These results all rely on a linearization of the fully nonlinear system to establish the embedding
$y$ for given metric $g_{i j}$ of the Riemannian manifold.

\subsection{Isometric embedding}\label{Embedding}
Let $(M,g)$ denote an $n$-dimensional Riemannian manifold with ascribed metric tensor $g$.
Suppose that manifold $(M,g)$ can be embedded globally into $\mathbb{R}^{m}$ (the term \textit{immersion} is used when the embedding is local).
As stated in \S \ref{Geometry}, this assumption implies that there exist a system of local coordinates $x=(x_1, x_2,\ldots, x_n)$ on $M$
and an embedding $y=(y_1, y_2, \ldots, y_m)$ such that $\partial_{i}y \cdot  \partial_{j}y =g_{i j}$ hold.

For an $n$-dimensional Riemannian manifold, the components of the corresponding metric tensor may be represented by
the $n \times n$ symmetric matrix
\begin{equation*}\left [\begin{array}{ccc}g_{1 1} & \cdots  & g_{1 n} \\
\vdots  & \ddots  & \vdots  \\
g_{n 1} & \cdots  & g_{n n}\end{array}\right ].
\end{equation*}
There are $J_n=\frac{n (n +1)}{2}$ entries on and above the diagonal,
and we conclude in general that the isometric embedding problem
has three cases: $m >J_n, \ m =J_n$, and $m <J_n$,
where $m$ is the number of unknowns $(y_{1},y_{2},\ldots  y_{m})$,
and $J_n$ is the number of equations.
The crucial number $J_n=\frac{n(n +1)}{2}$ is called the \emph{Janet dimension}.
Let $(M,g)$ be an $n$-dimensional Riemannian manifold with metric $g$,
and let  the $k$th covariant derivative be denoted by $\nabla _{k}$.
This derivative permits differentiation along the manifold.
For scalars $\phi$, vectors $\phi_{i}$ and rank-two
tensors $\phi_{i j}$, the covariant derivatives
are given respectively by
\begin{equation*}
\nabla _{k}\phi  = \partial _{k}\phi, \quad
\nabla _{k}\phi _{j} = \partial _{k}\phi _{j} -\Gamma _{j k}^{l} \phi _{l}, \quad
 \nabla _{k}\phi _{i j} = \partial _{k}\phi _{i j} -\Gamma _{i k}^{l} \phi _{l j} -\Gamma _{j k}^{l} \phi _{i l},
\end{equation*}
where the \emph{Christoffel symbols} are calculated from metric $g$ by the formula:
\begin{equation*}
\Gamma_{ij}^{k} =\frac{1}{2} g^{k l} \big(\partial_{i}g_{j l} + \partial_{j}g_{i l} - \partial_{l}g_{i j}\big).
\end{equation*}
The metric
tensor with components $g^{k l}$ (upper indices) is the inverse of that with
components $g_{ij}$ (lower indices),
so that $g^{k l} g_{p l} =\delta_{l}^{k}$,
where $\delta_{l}^{k}$ is the usual Kronecker delta.
The Kronecker deltas of upper and lower order are defined similarly.

The \emph{Riemann curvature tensor}, $R_{i j k}^{l}$, is defined in terms of
the Christoffel symbols by
\begin{equation*}
R_{i j k}^{l} = \partial_{j}\Gamma_{k i}^{l} - \partial_{k}\Gamma_{j i}^{l}
  +\Gamma_{j p}^{l} \Gamma_{k i}^{p} -\Gamma_{k p}^{l} \Gamma_{j i}^{p}.
\end{equation*}
By lowering
indices, we have the \emph{covariant Riemann curvature tensor}:
\begin{equation*}
R_{i j k l} =g_{i q} R_{j k l}^{q},
\end{equation*}
or
\begin{equation*}
R_{i j k l} =g_{i q} \big(\partial_{k}\Gamma_{l j}^{q} -\partial_{l}\Gamma_{k j}^{q}
   +\Gamma_{k p}^{q} \Gamma_{l j}^{p} -\Gamma_{l p}^{q} \Gamma_{k j}^{p}\big)
\end{equation*}
which~ is written as $Riem(g)$.

This tensor possesses the minor \emph{skew-symmetries}:
\begin{equation*}
R_{i j k l} = -R_{j i k l} = -R_{i j l k},
\end{equation*}
and the \emph{interchange}
(or major) symmetry $R_{i j k l} =R_{k l i j}$.
The cyclic interchange of indices leads to the \emph{first Bianchi identity}:
\begin{equation*}
R_{i j k l} +R_{i k l j} +R_{i l j k} =0,
\end{equation*}
as well as
the \emph{second Bianchi identity}:
\begin{equation*}
\nabla _{s}R_{i j k l} + \nabla _{k}R_{i j l s} + \nabla _{l}R_{i j s k} =0.
\end{equation*}

If we use the Ricci tensor
\begin{equation*}
R_{i k}:=g^{j l} R_{i j k l},
\end{equation*}
then the second Bianchi identity can be written as
\begin{equation*}
B i a n (g):=g^{a b} ( \nabla _{b}R_{a m} -\frac{1}{2}  \nabla _{m}R_{a b}) =0.
\end{equation*}
Of course, with usual raising and lowering of indices and the Ricci identity,
we have
\begin{equation*}
B i a n (g) = \nabla ^{a}R_{a m} -\frac{1}{2}  \nabla _{m}g^{a b} R_{a b}
= \nabla ^{a}(R_{a m} -\frac{1}{2} g_{a m} g^{a b} R_{a b}).
\end{equation*}
The quantity,
$R_{a m} -\frac{1}{2} g_{a m} g^{a b} R_{a b}$,
is the Einstein tensor.
Furthermore, the Ricci tensor is written in terms of metric $g$ by the formula:
\begin{equation*}
R_{i j} = \partial_{p}\Gamma_{i j}^{p} -\partial j \Gamma_{i p}^{p} +\Gamma_{i j}^{q} \Gamma_{p q}^{p}
-\Gamma_{i p}^{q} \Gamma_{j q}^{p}.
\end{equation*}
A \textit{necessary}
condition for the existence of an isometric embedding is that there exist
functions
$$
H_{i j}^{\mu } =H_{j i}^{\mu },\;\kappa _{\mu  i}^{\nu } = -\kappa _{\nu  i}^{\mu }, \qquad \,1 \leq i,j \leq n,\;n +1 \leq \mu,\nu  \leq m,
$$
such that the \textit{Gauss equation} holds:
\begin{equation}
\sum _{\mu  =n +1}^{m}\big(H_{i k}^{\mu } H_{j l}^{\mu } -H_{i l}^{\mu } H_{j k}^{\mu }\big)=R_{i j k l},
\end{equation}
along with the \textit{Codazzi equations}:
\begin{equation}
\partial _{k}H_{i j}^{\mu } +\kappa_{\nu  k}^{\mu } H_{i j}^{\nu } -\Gamma_{k i}^{p} H_{p j}^{\mu }
-\Gamma_{k j}^{p} H_{i p}^{\mu } - \partial_{j}H_{i k}^{\mu } -\kappa_{\nu  j}^{\mu } H_{i k}^{\nu }
+\Gamma_{j i}^{p} H_{p k}^{\mu } +\Gamma_{j k}^{p} H_{i p}^{\mu } =0,
\end{equation}
and the \textit{Ricci equations}:
\begin{equation}
K_{\mu  i j}^{\nu }:=\partial_{i}\kappa_{\mu  j}^{\nu } - \partial_{j}\kappa_{\mu  i}^{\nu }
  +\kappa_{\eta  i}^{\nu } \kappa_{\mu  j}^{\eta } -\kappa_{\eta  j}^{\nu } \kappa_{\mu  i}^{\eta }
  -g^{p q} \big(H_{i p}^{\mu } H_{j q}^{\nu } -H_{j p}^{\mu } H_{i q}^{\nu }\big) =0.
\end{equation}
The Ricci system can be expressed in covariant form by the addition and subtraction
of the term, $\Gamma_{i j}^{q} \kappa_{\mu q}^{\nu}$,
to obtain
\begin{equation*}
\nabla_{i}\kappa_{\mu  j}^{\nu } - \nabla_{j}\kappa_{\mu  i}^{\nu } +\kappa_{\eta  i}^{\nu } \kappa _{\mu  j}^{\eta }
 -\kappa_{\eta  j}^{\nu } \kappa_{\mu  i}^{\eta }
 =g^{p q} \big(H_{i p}^{\mu } H_{j q}^{\nu } -H_{j p}^{\mu } H_{i q}^{\nu }\big).
\end{equation*}

\begin{theorem}[Allendoerfer \cite{Allen}]
Suppose that there exist symmetric functions $H_{i j}^{\mu} =H_{j i}^{\mu}$ and anti-symmetric
functions $\kappa_{\mu  i}^{\nu }= -\kappa _{\nu  i}^{\mu },\ \text{}n +1 \leq \mu ,\nu  \leq m$,
such that the Gauss-Codazzi-Ricci equations are satisfied in a simply connected domain.
Then there exists an isometric embedding of the $n$-dimensional Riemannian manifold
{\rm (}with the second fundamental form $H_{ij}^\mu$ and the normal bundle $\kappa_{\nu i}^\mu${\rm )}
into $\mathbb{R}^{m}$.
\end{theorem}

This theorem shows that the solvability of the Gauss-Codazzi-Ricci equations is both necessary and sufficient for an isometric
immersion.

While the examination~ of the Gauss-Codazzi-Ricci system appears to be an appealing route
to the proof of local isometric embedding,
it is in fact the more direct route by using the embedding equations~  $ \partial _{i}y \cdot  \partial _{j}y =g_{i j}$
that has proven more successful for the $C^{\infty }$embedding problem.
The first such result was given in Bryant-Griffiths-Yang \cite{BGY}, and
more refined results are due to Nakamura-Maeda \cite{NakamuraMaeda86,NakamuraMaeda89}
and Poole \cite{Poole}.
See also Chen-Clelland-Slemrod-Wang-Yang \cite{CCSWY}
for an alternative, simpler proof of the Nakamura-Maeda theorem in \cite{NakamuraMaeda86,NakamuraMaeda89},
which we will use here and state it as follows.

\begin{theorem}[Nakamura-Maeda]\label{Nakamura-Maeda}
Let~$\left (M,g\right )$ be a  $C^{\infty}$ three-dimensional Riemannian manifold,
and let $p\in M$ be a point such that the Riemann curvature tensor {\rm (}as defined by~$g${\rm )} does not vanish.
Then there exists a local $C^{\infty}$ isometric embedding of
a neighborhood~ $U_{0}$ of~$p$ into~$\mathbb{R}^{6}$.
\end{theorem}

\subsection{Metric solvability for a prescribed Riemann curvature}

We will need the crucial result~ for the metric solvability for a prescribed Riemann curvature
due to DeTurck-Yang \cite{DeYang}:

\begin{theorem}[DeTurck-Yang \cite{NakamuraMaeda86,NakamuraMaeda89}] \label{T:DeYang}
Let $\widehat{R}$ be a non-degenerate tensor $(\widehat{R_{ijkl}})$ over
a three-dimensional manifold {\rm (}say an open set in the three-dimensional Euclidean space{\rm )}
which satisfies the first Bianchi identity.
Then, for any point on the manifold {\rm (}say a point in the open set{\rm )},
there exists a $C^{\infty }$ Riemannian metric $g$ such that system
\begin{equation}\label{Rg}
\textit{Riem}(g)=\widehat{R}
\end{equation}
 is satisfied in a neighborhood of the point.
 \end{theorem}

\begin{remark} We note that the proof of DeTurck-Yang in \cite{NakamuraMaeda86,NakamuraMaeda89}
also shows that metric g satisfies $g_{ij}=\delta_{ij}$ at the point noted
in Theorem \ref{T:DeYang}.
\end{remark}

\begin{remark}
The non-degeneracy of the tensor $\widehat{R}$ is equivalent to the non-singularity
of the matrix (which is also denoted by $\widehat{R}$):
$$\begin{bmatrix}
\widehat{R_{1212}} & \widehat{R_{1223}} &  \widehat{R_{1213}}  \\[1mm]
\widehat{R_{1223}} & \widehat{R_{2323}} & \widehat{R_{1323}} \\[1mm]
\widehat{R_{1213}} & \widehat{R_{1323}} & \widehat{R_{1313}}
\end{bmatrix}.$$
\end{remark}

\smallskip
\section{Continuum mechanics}

We consider the balance laws of mass and momentum in three-dimensional continuum mechanics.
Denote by $T$ the (symmetric) Cauchy stress tensor, and assume that the fields for velocity $u$,
Cauchy stress $T$,  and density $\rho$ are consistent with some specific constitutive equation
for a body and satisfy the balances of mass and linear momentum (satisfaction
of balance of angular momentum is automatic).
The equations for the balance of linear momentum in the spatial representation and balance
of mass are:
\begin{equation*}
\begin{split}
&\partial_{1}(\rho  u_{1}^{2} -T_{1 1}) + \partial_{2}(\rho  u_{1} u_{2} -T_{1 2})
  + \partial_{3}(\rho  u_{1} u_{3} -T_{1 3}) = - \partial_{t}(\rho  u_{1}),\\
&\partial_1(\rho  u_{2} u_{1} -T_{2 1}) + \partial_{2}(\rho  u_{2}^{2} -T_{2 2})
  + \partial_{3}(\rho  u_{2} u_{3} -T_{2 3}) = - \partial_{t}(\rho  u_{2}),\\
&\partial_1(\rho u_3u_1-T_{31})+\partial_2(\rho  u_{3} u_{2} -T_{3 2})
  + \partial _{3}(\rho  u_{3}^{2} -T_{3 3}) = - \partial_{t}(\rho  u_{3}),\\
&\partial_{1}(\rho  u_{1}) + \partial _{2}(\rho  u_{2}) + \partial_{3}(\rho  u_{3})
  = - \partial _{t}\rho,\\
&\rho  =\rho_{0} (\det F)^{ -1},
\end{split}
\end{equation*}
where $\rho_{0}$ is the density of the body in the reference configuration,
and $F$ is the deformation gradient of the current configuration
with respect to this reference.

\begin{example}
The following are some examples of the Cauchy stress tensor $T$:
\begin{enumerate}
\item Inviscid compressible fluid: $T_{i j} = -p (\rho ) \delta _{i j}$ (compressible fluid, the Euler
equations).

\smallskip
\item Inviscid incompressible fluid with constant (unit) density: $\rho  =1, T_{i j} = -p \delta _{i j}$,
which imply $\div\, u =0$, $\Delta p =-\div(\div (u \otimes u))$ (incompressible fluid, the Euler equations).

\smallskip
\item Viscous incompressible fluid with
constant (unit) density $\rho  =1$,
$T_{i j} = -p \delta_{ij} +2 \gamma  D_{ij}$, $ D_{i j} =\frac{1}{2}\big(\partial_{i}u_{j} + \partial_{j}u_{i}\big)$, $\gamma>0$,~ $\div\, u =0$,
~ which imply $ \Delta p =-\div(\div (u \otimes u))$ (the Navier-Stokes equations).

\smallskip
\item Neo-Hookean elasticity: $T =\rho  F F^\top$.
\end{enumerate}
\end{example}

\section{A map of continuum motion into geometric motion}

Given a local $C^{\infty }$ (space-time) solution of the continuum balance laws of mass and momentum,
we define the following quantities $\widehat{R_{ijkl}}$ via the first Bianchi identity and the relations:
\begin{eqnarray*}
&\widehat{R_{2323}} =\rho  u_{1}^{2} -T_{1 1},\quad &\widehat{R_{1313}} =\rho  u_{2}^{2} -T_{22},\\[1mm]
&\widehat{R_{1212}} =\rho  u_{3}^{2} -T_{33},\quad &\widehat{R_{3123}} =\rho  u_{1}u_{2} -T_{1 2},\\[1mm]
&\widehat{R_{1223}} =\rho u_1u_3-T_{13}, \quad &\widehat{R_{3112}} =\rho  u_{2}u_{3} -T_{2 3}.
\end{eqnarray*}
Write $\widehat{R}$ as the  $3\times 3$ symmetric matrix:
\begin{equation*}
\begin{split}
\widehat{R}=&
\begin{bmatrix}
\widehat{R_{1212}} & \widehat{R_{1223}} & - \widehat{R_{3112}} \\[1mm]
\widehat{R_{1223}} & \widehat{R_{2323}} & -\widehat{R_{3123}} \\[1mm]
-\widehat{R_{3112}} & -\widehat{R_{3123}} & \widehat{R_{1313}}
\end{bmatrix}\\[2mm]
=&\begin{bmatrix}
\rho  u_{3}^{2} -T_{33} & \rho  u_{1}u_{3} -T_{1 3} &  -\rho  u_{2}u_{3} +T_{2 3} \\[1mm]
\rho  u_{1}u_{3} -T_{1 3} & \rho  u_{1}^{2} -T_{11}  &  -\rho  u_{1}u_{2} +T_{1 2} \\[1mm]
 -\rho  u_{2}u_{3} +T_{2 3} &  -\rho  u_{1}u_{2} +T_{1 2} & \rho  u_{2}^{2} -T_{22}
 \end{bmatrix}.
\end{split}
\end{equation*}
Then system
$$
Riem(g)=\widehat{R}
$$
is a system of six equations in the six unknown components
of metric~$g$.
Furthermore, matrix $\widehat{R}$ is non-singular when $\det \widehat{R}$ is non-zero.
In this case, the  DeTurck-Yang theorem (Theorem \ref{T:DeYang})
yields the local existence of a $C^{\infty }$  metric
$g$ in space.
Moreover, matrix $\widehat{R}$ is positive definite when the quantities
$$\rho  u_{3}^{2} -T_{33}, \quad
\det \begin{bmatrix}
\rho  u_{3}^{2} -T_{33} & \rho  u_{1}u_{3} -T_{13}  \\[1mm]
\rho  u_{1}u_{3} -T_{13}  & \rho  u_{1}^{2} -T_{11}
\end{bmatrix}, \quad
\det \widehat{R}$$
are positive.

\begin{example}
For the Euler equations of either compressible or incompressible flow,
$$
\det \widehat{R}=p^2\big(p +\rho q^{2}\big), \quad  q^{2}=u_1^2+u_2^2+u_3^2.
$$
It is easy to see that $\widehat{R}$ is positive definite when $p$ is positive.
\end{example}

We can then state the following theorem.
\begin{theorem}\label{GeoMotion}
Let $\rho$, $u_{i}$, $T_{ij}$,  $i, j =1, 2, 3$,
be a local $C^{\infty}$ space-time solution to the balance laws of mass and momentum
with  non-zero $\det\widehat{R}$. Then we have the following{\rm :}

\smallskip
{\rm (a)} There is a local space-time Riemannian metric~$g$ in a neighborhood of the origin
that satisfies both system \eqref{Rg} with $g_{ij}=\delta_{ij}$ at the origin
and
the following system{\rm :}
\begin{equation}\label{e61}
\begin{split}
&\partial_{t}(\varrho u_{1})= -\Gamma _{12}^{\lambda }\widehat{R}_{\lambda 323}-\Gamma _{23}^{\lambda }(\widehat{R}_{12\lambda 3}
    +\widehat{R}_{312\lambda})-\Gamma _{13}^{\lambda }\widehat{R}_{2\lambda 23} -\Gamma _{22}^{\lambda }\widehat{R}_{31\lambda 3}
    -\Gamma_{33}^{\lambda }\widehat{R}_{122\lambda}=:A_1,\\
&\partial_{t}(\varrho u_{2})= -\Gamma _{13}^{\lambda }(\widehat{R}_{\lambda 123}+\widehat{R}_{3\lambda 12})
  -\Gamma_{11}^{\lambda }\widehat{R}_{3\lambda 23} -\Gamma _{21}^{\lambda }\widehat{R}_{\lambda 313}
  -\Gamma _{23}^{\lambda }\widehat{R}_{1\lambda 13} -\Gamma _{33}^{\lambda }\widehat{R}_{\lambda 112}=:A_2,\\
&\partial_{t}(\varrho u_{3}) = -\Gamma _{12}^{\lambda }(\widehat{R}_{1\lambda 23}+\widehat{R}_{31\lambda 2})
  -\Gamma_{11}^{\lambda }\widehat{R}_{\lambda 223} -\Gamma_{22}^{\lambda }\widehat{R}_{311\lambda }
  -\Gamma _{31}^{\lambda }\widehat{R}_{\lambda 212} -\Gamma_{32}^{\lambda }\widehat{R}_{121\lambda}=:A_3,
\end{split}
\end{equation}
which is abbreviated  as
$$(\partial_t(\varrho u_1),\partial_t(\varrho u_2),\partial_{t}(\varrho u_{3})) =\left (A_{1},A_{2},A_{3}\right ).$$
Moreover,  the balance of mass and momentum imply
\begin{equation}\label{e62}
\partial _{tt}\varrho  = -\partial _{i}A_{i}.
\end{equation}

\smallskip
{\rm (b)} There is a local $C^{\infty}$ space-time isometric embedding $y$ for the three-dimensional
Riemannian manifold $(M,g)$ into
$\mathbb{R}^{6}$ such that $\partial _{i}y \cdot  \partial _{j}y =g_{i j}$.

\smallskip
{\rm (c)} Conversely,  if there is a smooth metric $g$ and smooth continuum fields $\rho$, $u_{i}$, $T_{ij}$, $i,j =1, 2, 3$,
which  satisfy \eqref{Rg} and \eqref{e61}--\eqref{e62},  then the balance laws of mass and momentum are satisfied.
\end{theorem}

\begin{proof}
(a) Since $\widehat{R}$ is non-singular,
the DeTurck-Yang theorem (Theorem \ref{T:DeYang})
yields the existence of a metric $g$ satisfying \eqref{Rg}.
Then system \eqref{e61} follows from the second Bianchi identity
and the balance law of momentum for the continuum fields.

\smallskip
(b) The existence of an isometric embedding follows from the Nakamura-Maeda theorem (Theorem \ref{Nakamura-Maeda}).

\smallskip
(c) Apply the second Bianchi identity to $\widehat{R}$.
Then, from \eqref{Rg}, we obtain the system:
\begin{equation}\label{e63}
\begin{split}
&\partial_{1}(\rho  u_{1}^{2} -T_{1 1}) + \partial_{2}(\rho  u_{1} u_{2} -T_{1 2})
  + \partial_{3}(\rho  u_{1} u_{3} -T_{1 3})
=A_{1},\\
&\partial_1(\rho  u_{2} u_{1} -T_{2 1}) + \partial_{2}(\rho  u_{2}^{2} -T_{2 2})
 + \partial_{3}(\rho  u_{2} u_{3} -T_{2 3})
=A_{2},\\
&\partial_{1}(\rho  u_{3} u_{1} -T_{3 1}) + \partial _{2}(\rho u_{3} u_{2} -T_{3 2})
 + \partial_{3}(\rho  u_{3}^{2} -T_{3 3})
=A_{3}.
\end{split}
\end{equation}
From \eqref{e61}, we now recover the balance law of linear momentum for the continuum fields.
Finally, take the divergence of \eqref{e63} and use \eqref{e62} to see
$$
\partial_{t}( \partial_{1}(\rho u_{1}) + \partial_{2}(\rho u_{2}) + \partial_{3}(\rho u_{3}))
= -\partial _{tt}\rho.
$$
Hence, if the balance law of mass is initially satisfied, it is locally satisfied.
\end{proof}

\begin{remark}[Blow-up scenario] We note that, if a metric $g$ satisfies \eqref{Rg},
then substitution of the formula for $\widehat{R}$, which is given in terms of $(\rho, u_i, T_{ij})$,
into \eqref{e61} yields:

\smallskip
(i) a system of nonlinear essentially ordinary differential equations for the incompressible Euler equations;

\smallskip
(ii)  a system of weakly first-order quasilinear partial
differential equations for the incompressible Navier-Stokes equations.

\smallskip
These systems are non-local due to relations  \eqref{Rg} and $\Delta p= -\div (\div( u\otimes u))$.
Nevertheless, these systems provide what may be an avenue for proving finite-time blowup of smooth solutions.
\end{remark}

\section{Shearing motion}\label{shearing}

In \S 4, we have shown that, if $\widehat{R}$ is non-singular, there exists a map from the continuum flow
to the geometric motion.
This motivates the question as to what can be said in the case when $\widehat{R}$~ is singular.
In essence, there are two examples: One for the incompressible Euler equations, and the other for neo-Hookean elasticity
which have been provided in \cite{ACSW}.
For this reason, we will give only a short discussion for the first example, and the second example follows analogously.

For the incompressible Euler equations with $p\ge 0$ (by scaling for any lower bound), the singularity of $\widehat{R}$
means pressure $p =0$,
and hence we consider the steady flow:
$$u_1=u_{1}(x_{2}), \quad p =u_{2}=u_{3} =0.$$
The desired embedding is given by
$$
y_{1} =Bx_{2}, \quad y_{2} =Bx_{1}, \quad y_{3} =f(x_{2}), \quad y_{4} =Bx_{3}, \quad y_{5} =y_{6}=0,
$$
with
$$
\partial_{1}y =(0,B,0,0,0,0), \quad \partial_{2}y =(B,0,f^{\prime},0,0,0), \quad
 \partial _{3}y =(0,0,0,B,0,0),
$$
where $B$ is a positive constant.
This yields metric $g^{ \ast }$ with components
$$
g_{11}^{ \ast}=g^\ast _{33}=B^2, \quad g_{22}^{ \ast } =B^2+f^{ \prime 2}, \quad
g_{12}^{ \ast } =g^\ast_{13} =g_{23}^{ \ast } =0.
$$
An orthonormal set of normal vectors is given by
$$
N_{4}=\frac1{\sqrt{B^2+f^{\prime 2}}}(f^{ \prime },0, -B,0,0,0), \quad
  N_{5}=(0,0,0,0,1,0),\quad
  N_{6}=(0,0,0,0,0,1).
$$
Then a direct calculation by using the definitions of $H_{ij}^{\mu }$
and $\kappa_{\mu j}^{\upsilon }$ gives
$$
H_{22}^{4}= -\frac{f^{ \prime  \prime}}{\sqrt{B^2+f^{ \prime 2}}},
$$
and all the remaining components of the second fundamental forms, as well as
all $\kappa _{\mu j}^{\upsilon }$, to be zero.
The Gauss equations show that the manifold is Riemann flat, and the identification
$$
f^{ \prime }(x_2) =B\arctan(-B\int _{0}^{x_{2}}u_{1}^{2}(s)ds), \qquad
u_{1}^{2}(x_2) = -\frac{f^{ \prime  \prime}(x_2)}{\sqrt{B^2 +f^{\prime 2}(x_2)}}
$$
shows that the non-trivial Euler equation:
$$
0 = \partial _{1}u_{1}^{2}
$$
is identical to the Codazzi equation:
$$
0 = \partial _{1}H_{22}^{4}.
$$
Note the term, $\Gamma_{12}^2H_{22}^4$, in the Codazzi equation vanishes since $\Gamma_{12}^2=0$.

\begin{remark}
We note the following observation: If we had allowed the steady shearing motion
to be the more general case:
$$
u_1=u_{1}(x_{2},x_{3}), \quad  p =u_{2} =u_{3} =0,
$$
we still have an exact solution to the incompressible Euler equations.
However, perhaps surprisingly, we have not been able to find a three-dimensional
manifold that can be identified with this motion.
\end{remark}

\section{The Nash-Kuiper theorem and wild solutions}\label{NK-wild}

We wish to compare the metrics arising from the two cases: $p >0$ and $p =0$, for the incompressible Euler equations.
First consider metric $g$ given by Theorem \ref{GeoMotion} when $p >0$.
Since we know that there is an isometric embedding, we can expand the metric locally
via the Taylor series:
$$
g_{ij} =\delta_{ij}+ {\rm h.o.t}.
$$
For $B>0$ (in the definition of $g^*$) sufficiently large,
we see that $g \leq g^\ast$ in the sense of quadratic forms, or in the language of
the Nash-Kuiper theorem that the embedding $y_{g}$ is \textit{short} compared to the embedding $y_{g^{\ast}}$.

We now recall the well known Nash-Kuiper-Gromov results.

\begin{theorem}[Nash-Kuiper-Gromov]\label{Nash-Kuiper}
Let $\left (M^{n}, \mathfrak{g}\right )$ be a smooth, compact Riemannian manifold~$n \geq 2$.
Assume that $\mathfrak{g}$ is in~$C^{\infty}$. Then the following hold{\rm :}

\smallskip
{\rm (i)} If $m \geq \frac{1}{2}\left (n +2\right )\left (n +3\right )$,
any short embedding~$y_{\mathfrak{g}}$ into $\mathbb{R}^{m}$ can be approximated by isometric embeddings
into $\mathbb{R}^{m}$ of class  $C^{\infty }$ {\rm (Nash \cite{Nash1954} and Gromov \cite{Gromov1986})}{\rm ;}

\smallskip
{\rm (ii)} If~$m \geq n +1$, then any short embedding~$y_{\mathfrak{g}}$ into $\mathbb{R}^{m}$ can be approximated in $C^0$
by embeddings into $\mathbb{R}^{m}$~of
class $C^{1}$ {\rm (Nash \cite{Nash1954} and Kuiper \cite{Kuiper})}.
\end{theorem}

In our examples, we are in the case $n =3$ and $m =6$, so only case (ii) applies.
In fact, the issue is quite subtle. It is quite evident that the shearing motion had its geometric image
in  $\mathbb{R}^{4}$ and, if the general Euler flow had its geometric image in
$\mathbb{R}^{4}$ as well, then we could apply (ii) with $m =n +1$.
This would not only allow a sharper result but more importantly allow us to apply
the theory for $m =n +1$ by Borisov\cite{Borisov6,BorisovY},
at least in the case when the Euler flow is locally analytic, which states that embeddings $y_{\mathfrak{g}}$ of (ii)
cannot be $C^{2}$.
Thus there would be an upper bound on their regularity.
While it seems likely based on the discussion of  \cite{CDS2012}
that some upper bound regularity exists for our case, we know of no such result.

As a direct consequence of Theorem \ref{GeoMotion}, Theorem \ref{Nash-Kuiper},
and the appendix of \cite{ACSW},
we can state the following result.

\begin{theorem}\label{T62}
The short embedding $y_{g}$~ given Theorem {\rm \ref{GeoMotion}} {\rm (}which is the image of the Euler flow with $p>0${\rm )}
can be approximated
in $C^0$ by $C^{1}$ embeddings $y_w$
locally in space-time.
Furthermore, the energy
$$
E(t) =\int_{\Omega }\partial_i y_w \cdot \partial_{i}y_w\, dx =\int _{\Omega }tr(g^{*}) dx
$$
is a constant in $t$. In addition, $y_w$ is continuous in $t$ taking values in $C^1_{\rm loc}$.
\end{theorem}

Finally, we have the following non-uniqueness result.

\begin{theorem}\label{T63}
For fixed generalized shear initial data, there are infinite number of solutions to the initial value problem
for the evolution equations{\rm :} $\partial_{i}y \cdot \partial_{j}y =g_{i j}^{\ast}$.
\end{theorem}

The proof is identical to the one given in Theorem 7.3 in \cite{ACSW}.
Nevertheless, we provide a short sketch of the proof for the sake of completeness.

\begin{proof}
Choose an interval $[0,T]$ for which
the continuum mechanical balance laws have a smooth solution and a smooth embedding into $\R^6$ exists
via Theorem \ref{GeoMotion}.
Now take a sequence $\e_k>0$, $k=1, 2, \cdots, n$, and
$[T_0, T_1]$,  $[T_1, T_2]$,  $\cdots$,  $[T_{n-1}, T_n]$ with $T_0=0$ and $T_n=T$.
By Theorem \ref{T62}, there exists a sequence of wrinkled solutions $\{y_w^k\}$ so that
we can then define the wrinkled solution on the entire interval $[0,T]$ by
$$
y_w^\ast=y_w^k  \qquad \mbox{for $T_{k-1}\le t\le T_k, \, k=1, 2, \cdots, n$}.
$$
In addition, $y_w^*$ is continuous in $t$ with values in $C^1_{\rm loc}$
on each subinterval $[T_{k-1}, T_k]$.

Since $\|y_g-y_w^k\|_C\le\e_k$ on each subinterval $[T_{k-1}, T_k]$,
we can provide an infinite number of wrinkled solutions on $[0,T]$
by simply letting $\e_k$ and $T_k$ vary for $k=2, \cdots, n$, with $y_w^{k-1}$
as the fixed initial data.
The energy
$$
E(t)=\int_\O\partial_i y_w\cdot\partial_i y_wdx=\int_\O g_{ii}^\ast dx=const.
$$
on any domain $\O$ in $t>0$.
\end{proof}

Again as in \cite{ACSW},
it is not apparent in what sense $y_w$
actually satisfies
the Euler equations.
On the other hand, the simplicity of our arguments gives a rather elementary
explanation for both the existence of wild solutions to the equations of inviscid continuum
mechanics and the non-uniqueness of solutions to the Cauchy problem.

\medskip
\section{The Navier-Stokes equations and Couette flow}

In this section, we will apply our Nash-Kuiper approach to the classical Navier-Stokes equations
of viscous, incompressible fluid flow:
$$
T_{i j} = -p \delta_{ij} +2 \gamma  D_{ij},\quad  D_{i j} =\frac{1}{2}( \partial_{i}u_{j} + \partial_{j}u_{i}),
\quad
\gamma  >0, \quad {\rm div}\, u =0,
$$
which in turn imply
$$
\Delta p = -\div (\div (u \otimes u)).
$$

In particular, we will consider the problem of planar Couette flow between two parallel plates placed
at $x_{3} = \pm 1$ with the top (bottom) plate moving with speed $V$ (respectively $-V$).
We impose the non-slip boundary conditions:
$$
u_{1} = \pm V, \quad u_{2} =u_{3} =0 \qquad \mbox{at  $x_{3} = \pm 1$}.
$$
Define $\Omega $ to be the spatial domain:
$$
\Omega:=\{ (x_{1}, x_{2})\in\R^2 \;:\;  -1 \leq x_{3}\leq 1\}.
$$
We have taken~ dimensionless independent and dependent variables, and hence the quantity~$\gamma $
is now the inverse of the dimensionless Reynolds number.

An exact solution to the Navier-Stokes equations is given by the laminar Couette flow:
$$
u_{1} =Vx_{3}, \quad u_{2} =u_{3} =0, \quad p =const.
$$
and we take the constant to be zero so that $p =0$.
At first glance, there appears to be no external force for Couette flow, but of course this is not the case,
as an external force would be required to move the parallel plates and hence energy is being added to the system.
This fact is reflected in the fact that the $L^{2}\left (\Omega \right )$ energy norm of the laminar flow is infinite.
As in \S \ref{shearing}, matrix $\widehat{R}$ is singular, yet the laminar Couette flow can be identified
with a three-dimensional Riemannian manifold embedded in $\R^{6}$.

A crucial feature of planar Couette flow and related problems of Pouseille flow (pressure driven flow between
parallel plates) and rotating Couette flow (the fluid confined between two rotating cylinders)
is the numerically and physically observed bifurcation from laminar flow to roll patterns at high Reynolds number;
{\it cf.} \cite{CBusse, Nagata, Wal} and the references cited therein.
This means that the boundary value problem described above is expected to have non-unique solutions.
Hence, just as the simple shear solution for the Euler equations given in \S \ref{shearing}  provides an infinite number
of possible steady solutions to the Euler equations,
the boundary value problem for planar Couette flow can be expected to provide
a multitude of steady solutions at high Reynolds number.

By interchanging components for the embedding given in \S \ref{shearing},
the laminar Couette flow is embedded
into $\R^6$ with
$$
y_{1} =Bx_{3},\;\; y_{2} =f(x_{3}),\;\; y_{3} =Bx_{2},\;\; y_{4} =Bx_{1}, \;\; y_{5} =y_{6} =0,
$$
with
$$
\partial_{1}y =(0,0,0,B,0,0), \;\;  \partial_{2}y =(0,0,B,0,0,0), \;\;
\partial_{3}y =(B, f', 0, 0, 0, 0).
$$
This yields metric $g^{\ast}$ with components:
$$
g_{11}^{\ast} =g^\ast_{22}=B^2, \;\; g_{33}^{\ast} =B^2 +(f')^2, \;\;
g_{12}^{\ast} =g^\ast_{13} =g_{23}^{\ast} =0,
$$
which is Riemann flat.
Since we are using laminar Couette flow,
we have
$$
f'=B\arctan (-\frac{BV^2}3 x_3^3).
$$

Now we consider the initial value problem for the Couette flow.
Consider the initial data for $u$ with $\div\, u=0$ which satisfies the boundary conditions:
$$
u_{1} = \pm V,\quad  u_{2} =u_{3} =0  \qquad\; \mbox{at  $x_{3} = \pm 1$},
$$
and for $p$ initially as a solution
of the boundary value problem:
$$
\Delta p =-\div(\div (u \otimes u))
$$
to be positive on  $\Omega$.
This is easily done by adding a sufficiently large positive constant to any solution $p$ of
a fixed solution to the boundary value problem for $p$ in the bounded domain $\Omega$.
Hence, for sufficiently large Reynolds numbers, {\it i.e.},  small $\gamma $,
the contributions of the viscous stresses
$$
T_{i j} =2 \gamma  D_{ij}, \quad  D_{i j} =\frac{1}{2}( \partial _{i}u_{j} + \partial_{j}u_{i}), \quad \gamma  >0,
$$
to the computation of matrix
$\widehat{R}$ become negligible, and $\widehat{R}$ is initially positive definite.
In fact, this condition can be computed explicitly as follows:

The eigenvalues of matrix $\widehat{R}$ for the initial data $(u_{1}(x_{3}),0,0)$ and $p =p_{0}$ (constant) are given by
$$
\lambda_{1,2} =\frac{1}{2}\big(u_{1}^{2} +2p_{0} \pm \sqrt{u_{1}^{4} +4\gamma ^{2}(u_{1}^{\prime})^2}\big),
\quad \lambda_{3} =p_{0}.
$$
Hence, $p_{0}$ satisfying that $p_{0}^{2} +p_{0}u_{1}^{2}>\gamma^{2}(u_{1}^{\prime})^2$ would suffice,
so that  $\lambda_i>0$, $i=1,2,3$.
Here constant $p_{0}$ can be interpreted as the pressure at the ends of the channel.
Since the initial value problem for the Navier-Stokes equations has local-in-time smooth solutions,
we have a smooth solution with $\widehat{R}$ positive definite, locally in time.

We can now follow again the argument in \S \ref{NK-wild} to see that the geometric image
(at least locally in space-time) of the incompressible Navier-Stokes flow is approximated in $C^{0}$
by non-smooth $y_w$
and, furthermore, the initial value problem for $y_w$
has an infinite number of solutions. We state this as

\begin{theorem}
Theorems {\rm \ref{T62}}--{\rm \ref{T63}} are valid where

\begin{enumerate}
\item[\rm (a)] $y_{g}$ is the geometric image {\rm (}locally in space-time, say a domain $x\in \Omega_{1} \subset \Omega, \ 0 \leq t <T${\rm )}
of solutions of the initial value problem for the incompressible Navier-Stokes equations with the boundary
conditions:
$$
u_{1} = \pm V, \quad  u_{2} =u_{3} =0 \qquad \,\, \mbox{at  $x_{3} = \pm 1$};
$$

\smallskip
\item[\rm (b)]
$y_w\in  C^{1}(\Omega _{1})$  for each $t\in (0,T)$
and $y_w\in L^\infty((0,T); C^1_{\rm loc}(\Omega_{1}))$
satisfies $\partial_iy_w\cdot\partial_jy_w=g_{ij}^\ast$,  where $g^{\ast}$ has the components:
$$
g_{11}^{\ast} =g^\ast_{22}=B^2, \quad g_{33}^{\ast} =B^2+ (f^{\prime})^2, \quad  g_{12}^{\ast} =g^ \ast_{13} =g_{23}^{\ast} =0,
$$
which is Riemann flat.
\end{enumerate}
\end{theorem}

As noted above, if there was only one solution to the steady Couette problem,
the result of such non-uniqueness would seem unlikely,
but
the key here is that there are many such solutions at large Reynolds number
and the non-uniqueness of the evolutionary geometric problem is not unexpected.

Of course, these results give  an indication at the geometric level for both the~ onset of turbulence
for high Reynolds number Navier-Stokes flow and non-uniqueness of~ weak solutions for the Cauchy  problem
for the three-dimensional Navier-Stokes equations.
Specifically, we note that, to apply the above theory, we have needed the three eigenvalues  at the ends of
the channel fixed and positive, and the transition from the wild ``turbulent'' initial data
will occur as $\lambda_{2}$ passes from positive to negative values for decreasing
Reynolds number $\gamma^{-1}$.
Thus, according to our geometric theory, the critical Reynolds number will be given by the formula:
$$
 p_{0}^{2} +p_{0}u_{1}^{2} =\gamma_{\rm crit}^{2}(u_{1}^{\prime})^2.
$$
In addition, this formula provides an equation for the critical profile for transition from turbulence.
One possible check of this critical profile relation is to use the experimental data
of Reichardt \cite{Reichardt} where the data are normalized so that $V =1$.

Figure (a) (from \cite{Reichardt}) gives profiles for water with Reynolds number $Rw=18000$ and oil with Reynolds number $Ro=2900$.
We consider the one for oil since the Reynolds number for that experiment is closer to the usually accepted Reynolds number
  $\sim 2300$ for transition to turbulence.
Since we expect that $p_{0} \ll 1$, we drop that term in our critical profile equation and simplify it
as
$$
p_{0}u_{1}^{2} =\gamma_{\rm crit}^{2}(u_{1}^{\prime})^2.
$$
Let us work on the interval $0 \leq x_{3} \leq 1$.
We can find the solution for $ -1 \leq x_{3} \leq 0$ via the relation,
  $u_{1}( -x_{3}) = -u(x_{3})$.
Hence, for  $0 \leq x_{3} \leq 1$,  we solve the initial value problem:
$$
u_{1}^{ \prime}=au_1, \qquad     u_{1}(1) =1,
$$
for $a =\frac{\sqrt{p_{0}}}{{\gamma}_{crit}}$
to give
$$
u_{1}(x_{3})=\exp (a(x_{3} -1)).
$$
If we fit this relation to Reichardt's Figure (a)  by using $u_{1}(.8)=.4$,
we find $p_{0}=(.00159)^{2}$, $a =(.00159..)(2900) =4.611$,
and so the critical profile is approximately given by $u_{1}(x_{3})=\exp (4.611(x_{3} -1))$.

We note the value of our approximate critical profile at $x_{3} =0$ is $u_{1}(0 +)=\exp(-4.611)$.
While this value is not identically zero, its value is sufficiently small to provide  a very good approximation
to the Reichardt's graph which has the value zero at  $x_{3} =0$.
We plot our approximate critical profile in Figure (b)
with the $x$--axis corresponding to $x_{3}$ and the $y$--axis corresponding
to $e^{4.611(x -1)}$.

\begin{center}
\includegraphics[height=1.8in]{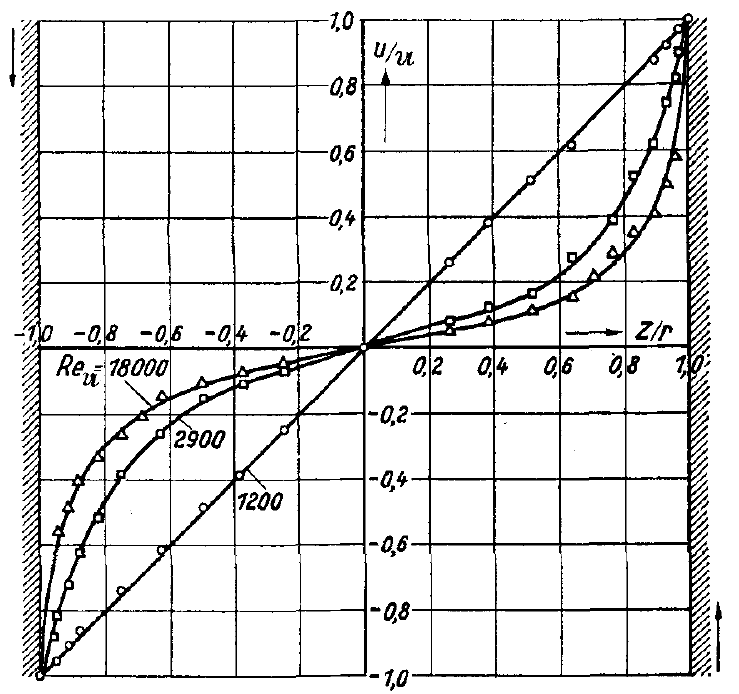}\qquad
\includegraphics[height=1.8in]{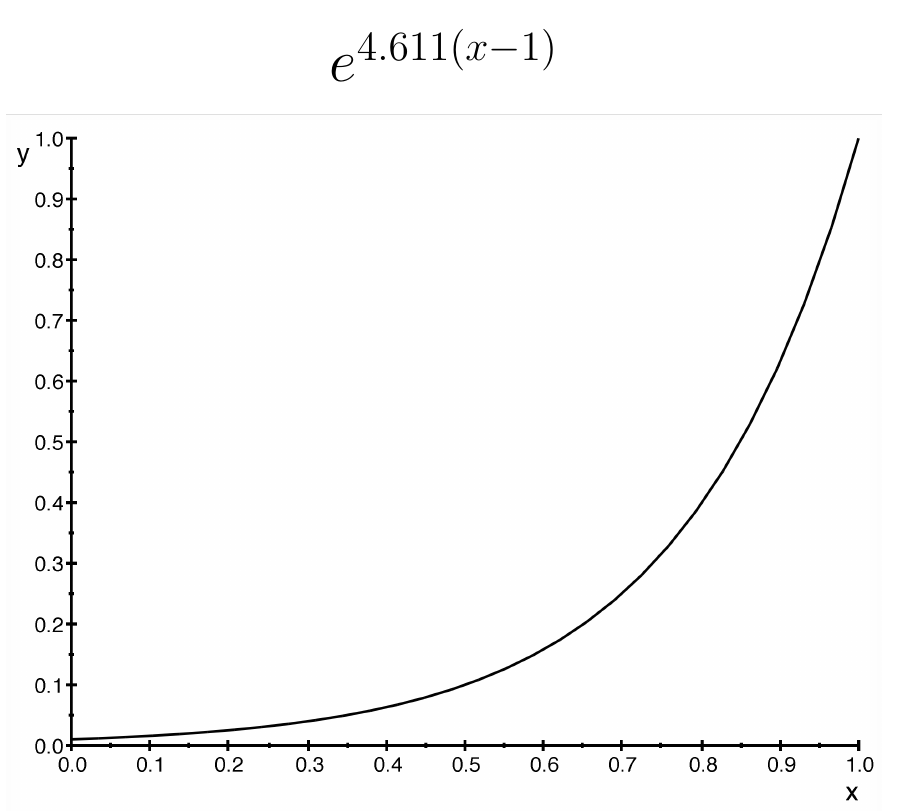}

(a) \hskip 5cm (b)
\end{center}

While the above computation does not validate the geometric theory,
it does at least show the geometric theory is consistent with experiment.
We note that the experiments of Cadot et al \cite{Cadot} have also demonstrated
the occurrence of low pressure turbulence.

\bigskip

\section*{Acknowledgments}
G.-Q. Chen's research was supported in part by
the UK
Engineering and Physical Sciences Research Council under Grants
EP/E035027/1 and EP/L015811/1,
and the Royal Society Wolfson Research Merit Award (UK).
M. Slemrod was supported in part by Simons Collaborative Research Grant 232531.
D. Wang was supported in part by NSF grants DMS-1312800 and DMS-1613213.
We also wish to thank Amit Acharya, Siran Li, and Laszlo Sz\'{e}kelyhidi for their valuable
remarks on this research project.


\bigskip

\begin{thebibliography}{99}

\bibitem{ACSW}
A. Acharya, G.-Q. Chen, S. Li, M. Slemrod, and D. Wang,
{\em  Fluids, elasticity, geometry, and the existence of wrinkled solutions}.
Arch. Ration. Mech. Anal. 2017,
doi: 10.1007/s00205-017-1149-5,  arXiv:1605.03058 [math.AP].

\bibitem{Allen}
C. B. Allendoerfer,
{\em The imbedding of Riemann spaces in the large.}
 Duke Math. J. 3 (1937), 317--333.

\bibitem{Blum}
R. Blum,
{\em Subspaces of Riemannian spaces.}
 Canad. J. Math. 7 (1955), 445--452.

\bibitem{Borisov6}
J. F.  Borisov,
{\em  $C^{1,\alpha }$--isometric immersions of Riemannian
spaces.} Doklady 163 (1965), 869--871.

\bibitem{BorisovY}
Y.  Borisov,
{\em  Irregular $C^{1,\beta}$-surfaces with analytic metric.}
Sib. Mat. Zh. 45, 1 (2004), 25--61.


\bibitem{BGY}
R. L. Bryant, P.  A.  Griffiths, and D. Yang,
{\em Characteristics and existence of isometric embeddings.}
 Duke Math. J. 50 (1983),
 893--994.

\bibitem{Cadot}
O. Cadot, S. Douady, and Y. Couder,
{\em Characterization of the low-pressure filaments in a three-dimensional turbulent
shear flow}.
Phys. Fluids, 7 (3) (1995), 630--646.




\bibitem{CCSWY}
G.-Q.  Chen, J. Clelland, M. Slemrod, D. Wang, and D. Yang,
{\em Isometric embedding via strongly symmetric positive systems.}
Asian J. Math. 2017 (to appear).  arXiv:1502.04356 [math.DG].


\bibitem{CDK}
E. Chiodaroli, C. De Lellis, and O. Kreml,
{\em Global ill-posedness of the isentropic system of gas dynamics.}
 Comm. Pure Appl. Math. 68 (2015),
 1157--1190.

\bibitem{CBusse}
R. M. Clever and F. H. Busse,
{\em Tertiary and quaternary solutions for plane Couette flow.}
J. Fluid Mech. 344 (1997), 137--153.

\bibitem{CDS2012}
S. Conti, C.  De Lellis, and L. Sz\'{e}kelyhidi,
{\em h-principle and rigidity for $C^{1,\alpha }$ isometric embeddings.}
Nonlinear partial differential equations, 83--116, Abel Symp., 7, Springer,
Heidelberg, 2012.

\bibitem{DS2009}
C. De Lellis and L. Sz\'{e}kelyhidi,
{\em The Euler equations as a differential inclusion.}
Ann. of Math. (2), 170 (2009),
1417--1436.

\bibitem{DS2010}
C. De Lellis and L. Sz\'{e}kelyhidi,
{\em On admissibility criteria for weak
solutions of the Euler equations.}
Arch. Ration. Mech. Anal. 195 (2010),
225--260.

\bibitem{DS2012}
C. De Lellis and L. Sz\'{e}kelyhidi,
{\em The h-principle and the equations of fluid dynamics.}
 Bull. Amer. Math. Soc. (N.S.), 49 (2012),
 347--375.


\bibitem{DeYang}
D. DeTurck and D. Yang,
{\em Local existence of smooth metrics with
prescribed curvature.} Nonlinear Problems in Geometry (Mobile, Ala., 1985),
37--43, Contemp. Math. 51, Amer. Math. Soc.: Providence, RI, 1986.

\bibitem{GoodmanYang}
J. Goodman and D. Yang.
{\em Local solvability of nonlinear partial differential equations
of real principal type.}  Unpublished notes.


\bibitem{Gromov1986}
M. Gromov, {\em Partial Differential Relations.}
Springer-Verlag: Berlin, 1986.

\bibitem{HanHong}
Q. Han and J.-X.  Hong,
{\em  Isometric Embedding of Riemannian Manifolds
in Euclidean Spaces.}  Mathematical Surveys and Monographs, 130.
Amer. Math. Soc.:  Providence, RI, 2006.

\bibitem{HanKhuri}
Q. Han and M. Khuri,
{\em The linearized system for isometric embeddings and its characteristic variety.}
Adv. Math. 230 (2012),
263--293.


\bibitem{Kuiper}
N. H. Kuiper, {\em On $C^1$--isometric imbeddings. I, II.}  Nederl. Akad.
Wetensch. Proc. Ser. A. 58 = Indag. Math. 17 (1955), 545--556, 683--689.


\bibitem{Nagata}
M. Nagata,
{\em A note on the mirror-symmetric coherent structure in plane Couette flow.}
 J. Fluid Mech. 727 (2013), R1--8.

\bibitem{NakamuraMaeda86}
G. Nakamura and Y.  Maeda,
{\em Local isometric embedding problem of Riemannian 3-manifold into~$R^{6}$. }
Proc. Japan Acad. Ser. A Math. Sci. 62 (1986),
257--259.

\bibitem{NakamuraMaeda89}
G. Nakamura and Y.  Maeda,
{\em Local smooth isometric embeddings of low-dimensional Riemannian manifolds into Euclidean spaces.}
 Trans. Amer. Math. Soc. 313 (1989),
 1--51.


\bibitem{Nash1954}
J. Nash, {\em $C^1$ isometric imbeddings.}
 Ann. of Math. (2), 60 (1954), 383--396.


\bibitem{Poole}
T. E.  Poole,
{\em  The local isometric embedding problem for $3$-dimensional Riemannian manifolds with cleanly vanishing curvature.}
Comm. Partial Diff. Eqs. 35 (2010),
1802--1826.

\bibitem{Reichardt}
H. Reichardt,
{\em Uber die Geschwindigkeitsverteilung in einer geradlinigen turbulenten Couettestrornung.}
 Z. Angew. Math. Mech. 96 (1956), S26--29.


\bibitem{Riemann}
B. Riemann,
{\em Ueber die Hypothesen, welche der Geometrie zu Grunde liegen}.
In: {\em Habilitationsschrift, 1854, Abhandlungen der K\"{o}niglichen Gesellschaft der Wissenschaften zu G\"{o}ttingen},
13 (1868), S. 133--150.

\bibitem{Wal}
F. Waleffe,
{\em Homotopy of exact coherent structures in plane shear flows.}
Physics of Fluids, Vol. 15  (2003), no. 6.

\bibitem{WallNagata}
D. P. Wall and M. Nagata, {\em Exact coherent states in channel flow.}
 J. Fluid Mech. 788 (2016),  444--468.

\end{thebibliography}
\end{document}